\documentclass[12pt]{amsart}
\usepackage{amsfonts,latexsym,amscd}

\usepackage[all,cmtip,2cell]{xy}
\UseAllTwocells
\usepackage{verbatim}
\usepackage{hyperref}
\usepackage{mathrsfs}
\usepackage{amssymb}

\numberwithin{equation}{section}

\newtheorem{theorem}{Theorem}[section]
\newtheorem{lemma}[theorem]{Lemma}
\newtheorem{proposition}[theorem]{Proposition}
\newtheorem{corollary}[theorem]{Corollary}
\theoremstyle{definition}

\newtheorem{example}[theorem]{Example}

\newtheorem{remark}[theorem]{Remark}

\newcommand{\Bnd}{{\rm End}}

\newcommand{\FPdim}{{\rm FPdim}}

\newcommand{\Hom}{{\rm Hom}}

\newcommand{\Id}{{\rm Id}}

\newcommand{\Aut}{\text{Aut}}

\newcommand{\Rep}{{\rm Rep}}

\newcommand{\Coh}{{\rm Coh}}
\newcommand{\Corep}{\rm Corep}
\newcommand{\Vect}{{\rm Vec}}

\newcommand{\B}{\mathscr{B}}
\newcommand{\C}{\mathscr{C}}
\newcommand{\D}{\mathscr{D}}
\newcommand{\M}{\mathscr{M}}
\newcommand{\N}{\mathscr{N}}

\renewcommand{\O}{\mathscr{O}}
\newcommand{\A}{\mathscr{A}}

\newcommand{\g}{\mathfrak{g}}

\newcommand{\ot}{\otimes}

\newcommand{\ben}{\begin{enumerate}}
\newcommand{\een}{\end{enumerate}}

\newcommand{\End}{{\rm End}}

\newcommand{\E}{\mathscr{E}}
\newcommand{\Z}{\mathscr{Z}}

\newcommand{\Mod}{{\rm Mod}}
\newcommand{\Bimod}{{\rm Bimod}}

\begin{document}

\title[Minimal extensions of Tannakian categories] {Minimal extensions of Tannakian categories in positive characteristic}

\author{Shlomo Gelaki}
\address{Department of Mathematics, Iowa State University, Ames, IA 50011, USA} 
\email{gelaki@iastate.edu}

\date{\today}


\begin{abstract}
We extend \cite[Theorem 4.5]{DGNO} and \cite[Theorem 4.22]{LKW} to positive characteristic (i.e., to the finite, not necessarily fusion, case). Namely, we prove that if $\D$ is a finite non-degenerate braided tensor category over an algebraically closed field $k$ of characteristic $p>0$, containing a Tannakian Lagrangian subcategory $\Rep(G)$, where $G$ is a finite $k$-group scheme, then $\D$ is braided tensor equivalent to $\Rep(D^{\omega}(G))$ for some $\omega\in H^3(G,\mathbb{G}_m)$, where $D^{\omega}(G)$ denotes the twisted double of $G$ \cite{G2}. We then prove that the group $\mathcal{M}_{{\rm ext}}(\Rep(G))$ of minimal extensions of $\Rep(G)$ is isomorphic to the group $H^3(G,\mathbb{G}_m)$. In particular, we use \cite{EG2,FP} to show that $\mathcal{M}_{\rm ext}(\Rep(\mu_p))=1$, $\mathcal{M}_{\rm ext}(\Rep(\alpha_p))$ is infinite, and if $\O(\Gamma)^*=u(\g)$ for a semisimple restricted $p$-Lie algebra $\g$, then $\mathcal{M}_{\rm ext}(\Rep(\Gamma))=1$ and $\mathcal{M}_{\rm ext}(\Rep(\Gamma\times \alpha_p))\cong \g^{*(1)}$.
\end{abstract}

\maketitle

\tableofcontents

\section{Introduction}

Let $\E$ be a finite symmetric tensor category over an algebraically closed field $k$ with characteristic $0$. A {\em minimal extension} of $\E$ is, by definition, a finite non-degenerate braided tensor category $\D$ over $k$, containing $\E$ as a Lagrangian subcategory (see \ref{Centralizers and Lagrangian subcategories} below). Given $\E$ as above, it is natural to try to classify its minimal extensions. Indeed, this problem was studied in several papers, e.g., \cite{DN,DGNO,DGNO2,DMNO,GS,GV,JFR,LKW,VR,OY}. 

In particular, in the papers \cite{DGNO,DGNO2,LKW} the authors classify the minimal extensions of Tannakian categories $\E=\Rep(G)$ over $k$ ($G$ a finite group). Note that in this case, $\E$ and all its minimal extensions are fusion categories. More explicitly, in \cite[Theorem 4.5]{DGNO}, \cite[Theorem 4.64]{DGNO2}, it is proved that $\D$ is a minimal extension of $\Rep(G)$ if and only if $\D$ is braided tensor equivalent to the representation category $\Rep(D^{\omega}(G))$ of a certain quasi-Hopf algebra $D^{\omega}(G)$ for some $\omega\in H^3(G,k^{\times})$. Then in \cite[Theorem 4.22]{LKW}, it is shown that minimal extensions of $\Rep(G)$ carry a natural structure of a commutative group, denoted by $\mathcal{M}_{\rm ext}(\Rep(G))$, and that $\mathcal{M}_{\rm ext}(\Rep(G))\cong H^3(G,k^{\times})$ as groups.

Our goal in this paper is to extend the above mentioned results of \cite{DGNO,DGNO2,LKW} to minimal extensions of Tannakian categories in positive characteristic. 

Namely, assume from now on that $k$ is an algebraically closed field with characteristic $p>0$, and let $G$ be a finite $k$-group scheme. Then our goal is to relate minimal extensions of $\Rep(G)$ over $k$, and representation categories of twisted doubles $D^{\omega}(G)$, $\omega\in H^3(G,\mathbb{G}_m)$ \cite{G2}.

More precisely, we first prove the following result.

\begin{theorem}\label{mainnew}
Assume $\D$ is a finite non-degenerate braided tensor category over $k$ containing a Tannakian Lagrangian subcategory $\Rep(G)$, where $G$ is a finite group scheme over $k$. Then  
there exists an element $\omega\in H^3(G,\mathbb{G}_m)$ such that $\D\cong \Z({\rm Coh}(G,\omega))$ as braided tensor categories (equivalently, $\D$ is braided tensor equivalent to the representation category $\Rep(D^{\omega}(G))$ of a twisted double $D^{\omega}(G)$ of $G$).
\end{theorem}

Fix a finite non-degenerate braided tensor category $\D$ over $k$, and consider the set of all triples
$(G, \omega, {\rm F})$, where $G$ is a finite group scheme over $k$, $\omega\in H^3(G,
\mathbb{G}_m)$, and ${\rm F}: \D \xrightarrow{\cong} \Z({\rm Coh}(G,\omega))$ is a
braided tensor equivalence. We say that $(G_1,
\omega_1, {\rm F}_1)$, $(G_2, \omega_2, {\rm F}_2)$ are equivalent, if there
exists a tensor equivalence $\varphi: {\rm Coh}(G_1,\omega_1)
\xrightarrow{\cong} {\rm Coh}(G_2,\omega_2)$ such that
$\mathcal{F}_2\circ {\rm F}_2  = \varphi\circ \mathcal{F}_1\circ {\rm F}_1$,
where $\mathcal{F}_i : \Z({\rm Coh}(G_i,\omega_i)) \twoheadrightarrow
{\rm Coh}(G_i,\omega_i)$ are the canonical forgetful
functors, $i=1,2$. Let $[(G,\omega,{\rm F})]$ denote the equivalence class of $(G,\omega,{\rm F})$.

Then using Theorem \ref{mainnew}, we obtain the following corollary.

\begin{corollary}\label{main1}
Let $\D$ be a finite non-degenerate braided tensor category over $k$. Then the following sets are in natural one to one correspondence with each other:
\begin{enumerate}
\item
Equivalence classes $[(G,\omega,{\rm F})]$.
\item
Tannakian Lagrangian subcategories of $\D$.
\end{enumerate}
\end{corollary}

Next we fix a finite group scheme $G$ over $k$, and consider the set of pairs $(\D,\iota)$, where $\D$ is a non-degenerate braided tensor category over $k$, and $\iota:\Rep(G)\xrightarrow{1:1} \D$ is an injective braided tensor functor. We say that two pairs $(\D_1,\iota_1)$ and $(\D_2,\iota_2)$ are equivalent, if there
exists a braided tensor equivalence $\psi:\D_1\xrightarrow{\cong}\D_2$ such that
$\psi\circ \iota_1= \iota_2$.
Let $\mathcal{M}_{\rm ext}(\Rep(G))$ denote the set of all equivalence classes of pairs $(\D,\iota)$. Recall that $\mathcal{M}_{\rm ext}(\Rep(G))$ has a natural structure of a commutative group \cite{LKW}. 

Then using Theorem \ref{mainnew}, we obtain the following corollary, which extends \cite[Theorem 4.22]{LKW} to positive characteristic.

\begin{corollary}\label{maincor}
Let $G$ be a finite group scheme over $k$. There is a group isomorphism 
$$\mathcal{M}_{\rm ext}(\Rep(G))\cong H^3(G,\mathbb{G}_m).$$
\end{corollary}

The structure of the paper is as follows. In Section \ref{prelim} we recall some necessary background on finite tensor categories and their exact module categories (\ref{FTC}, \ref{exmodcat}), exact sequences of finite tensor categories (\ref{esftc}), non-degenerate braided tensor categories and exact commutative algebras in finite braided tensor categories (\ref{Centralizers and Lagrangian subcategories}, \ref{Exact commutative algebras}), finite group schemes, de-equivariantization, the finite tensor categories ${\rm Coh}(G,\omega)$ and their centers $\Z({\rm Coh}(G,\omega))$, and the group $\mathcal{M}_{\rm ext}(\Rep(G))$ (\ref{gsch} - \ref{mextrepg}).
  
Sections \ref{prrofmainnew}, \ref{prrofmain1} and \ref{proofmaincor} are devoted to the proofs of Theorem \ref{mainnew}, Corollary \ref{main1} and Corollary \ref{maincor}, respectively. 

Finally, in Section \ref{examples} we give some examples. In particular, we use the results of \cite{EG2} to show that the group $\mathcal{M}_{\rm ext}(\Rep(\mu_p))$ is trivial (see Example \ref{exampl1}) and the group $\mathcal{M}_{\rm ext}(\Rep(\alpha_p))$ is infinite (see Example \ref{exampl2}), and use \cite{FP} to conclude that if $\O(\Gamma)^*=u(\g)$ for a semisimple restricted $p$-Lie algebra $\g$, then $\mathcal{M}_{\rm ext}(\Rep(\Gamma))$ is the trivial group and $\mathcal{M}_{\rm ext}(\Rep(\Gamma\times \alpha_p))\cong \g^{*(1)}$ (see Examples \ref{exampl3} and \ref{exampl4}).

\bigskip
\noindent {\bf Acknowledgment.} I am grateful to Pavel Etingof for reading the manuscript and suggesting Example \ref{exampl4}.

\section{Preliminaries}\label{prelim}
Let $k$ be an algebraically closed field with positive characteristic $p>0$, and let $\Vect$ denote the category of finite dimensional $k$-vector spaces.

Let $\B$ be any $k$-linear category. Recall that for every $V\in {\rm Vec}$, we have a natural functor 
$$V\ot - :\B\to \B,\,\,\,X\mapsto V\ot X.$$ Namely, for 
$X\in \B$, the object $V\ot X$ is uniquely defined through the Yoneda lemma by the formula 
$$\Hom_{\B}(Y,V\ot X)\cong V\ot _k\Hom_{\B}(Y,X)$$
(and the existence of this object is checked by choosing a basis in $V$).

\subsection{Finite tensor categories}\label{FTC} (See \cite[Chapter 6]{EGNO}.)  
Let $\B$ be a {\em finite} tensor category over $k$. Let $\mathcal{O}(\B)$ denote the complete set of isomorphism classes of simple objects of $\B$. 
Let ${\rm Gr}(\B)$ denote the Grothendieck ring of $\B$, and let ${\rm Pr}(\B)$ denote
the group of isomorphism classes of projective objects in $\B$. Recall that ${\rm Pr}(\B)$ is a bimodule over ${\rm Gr}(\B)$, and we have a natural homomorphism $\tau:{\rm Pr}(\B)\to {\rm Gr}(\B)$.

Recall \cite[Subsection 2.4]{EO} that we have a character 
$$\FPdim: {\rm Gr}(\B)\to \mathbb{R},\,\,\,X\mapsto \FPdim(X),$$ attaching to $X\in \B$ its {\em Frobenius-Perron dimension}. 
Recall also that 
$$\FPdim(\B):=\sum_{X\in \mathcal{O}(\B)}\FPdim(X)\FPdim(P(X)),$$ 
where $P(X)$ denotes the projective cover of $X$ and $\FPdim(P(X))$ is defined to be $\FPdim(\tau(P(X)))$.

Finally, recall that $\A$ is a {\em tensor subcategory} of $\B$ if $\A$ is a full subcategory of $\B$, closed under taking subquotients, tensor products, and duality \cite[Definition 4.11.1]{EGNO}.

\subsection{Exact module categories}\label{exmodcat}
Retain the notation from \ref{FTC}. Recall that a left $\B$-module category $\M$ is said to be {\em indecomposable} if it is not a direct sum of two nonzero module categories, and {\em exact} if $P\ot M$ is projective in $\M$, for every $P\in {\rm Pr}(\B)$ and $M\in\M$. 

Let $\mathscr{M}$ be an indecomposable exact 
$\B$-module category\footnote{Here and below, by ``a module category'' we will mean a {\em left} module category, unless otherwise specified.} (such a category is always finite \cite[Lemma 3.4]{EO}). 
Let $\Bnd(\M)$ be the abelian category of {\em right exact} endofunctors of
$\mathscr{M}$, and let $\B_{\mathscr{M}}^*:=
\Bnd_{\B}(\mathscr{M})$ be the dual category of $\B$ with respect to $\M$, i.e., the category of  
$\B$-linear (necessarily exact) endofunctors of $\mathscr{M}$. Recall that
composition of functors turns $\Bnd(\M)$ into a {\em monoidal}
category, and $\B_{\mathscr{M}}^*$ into a finite tensor category.

Let $A$ be an algebra object in $\B$, and let $\mathscr{M}:={\rm Mod}(A)_{\B}$ denote the category of {\em right} $A$-modules in $\B$. Recall that $\mathscr{M}$ has a natural structure of a left $\B$-module category, and $A$ is called {\em exact} if $\mathscr{M}$ is indecomposable and exact over $\B$ (see \cite[Section 7.5]{EGNO}). Recall  \cite[Proposition 7.11.6]{EGNO} that in this case, $\B_{\mathscr{M}}^*\cong{\rm Bimod}_{\B}(A)^{{\rm op}}$ is a finite tensor category (where ${\rm Bimod}_{\B}(A)^{{\rm op}}$ is ${\rm Bimod}_{\B}(A)$ equipped with the opposite tensor product).

The following result will be crucial for us.

\begin{theorem}\label{exactalg} \cite[Corollary 12.4]{EtO}
If $\C$ is a tensor subcategory of $\B$, and $A$ is an exact algebra in $\C$, then $A$ is an exact algebra in $\B$. \qed
\end{theorem}

Let $\mathscr{M}$ be an exact indecomposable module category over $\B$. Fix a nonzero object $M\in \mathscr{M}$, and let $A:=\underline{\Bnd}_{\B}(M)$ be the internal Hom from $M$ to $M$ in the $\B$-module category $\mathscr{M}$. Recall that $A$ is an exact algebra in $\B$, and there is a canonical equivalence of $\B$-module categories $\mathscr{M}\cong \Mod(A)_{\B}$ (see, e.g., \cite[Theorem 7.10.1]{EGNO}).

\subsection{Exact sequences of finite tensor categories}\label{esftc} 
Let $\A,\B$ and $\C$ be finite tensor categories, and let $\iota:\A\xrightarrow{1:1}\B$ be an injective tensor functor. Let $\M$ be an indecomposable exact module category over $\A$, fix a nonzero object $M\in \mathscr{M}$, and consider the algebrta $A:=\underline{\Bnd}_{\A}(M)$ (see \ref{FTC}, \ref{exmodcat}). We have $\mathscr{M}=\Mod(A)_{\A}$, and $\B\boxtimes_{\A} \mathscr{M}=\Mod(A)_{\B}$.

Let $F:\B\to \C\boxtimes \Bnd(\M)$ be a surjective \footnote{I.e., any object of $\C\boxtimes \Bnd(\M)$ is a subquotient of $F(X)$ for some $X\in \B$.} monoidal functor, such that 
$\iota(\A)\subset \B$ is the subcategory consisting of all objects $X\in \B$ such that $F(X)\in \Bnd(\M)$.
Recall \cite{EG} that $F$ defines an {\em exact sequence of tensor categories
\begin{equation}\label{ES}
\A\xrightarrow{\iota} \B\xrightarrow{F} \C\boxtimes \Bnd(\M)
\end{equation}
with respect to $\mathscr{M}$},
if for every object $X\in \B$ there exists a subobject $X_0\subset X$ such that $F(X_0)$ 
is the largest subobject of $F(X)$ contained in $\Bnd(\M)\subset \C\boxtimes \Bnd(\M)$. Recall \cite[Theorem 2.9]{EG} that the functor $F$ induces on $\C\boxtimes \M$ a structure of an exact $\B$-module category.

\begin{theorem}\label{eqdefexseq} \cite[Theorems 3.4, 3.6]{EG}
The following are equivalent:
\begin{enumerate}
\item
(\ref{ES}) defines an exact sequence of finite tensor categories.
\item
$\FPdim(\B)=\FPdim(\A)\FPdim(\C)$.
\item
The natural functor $T: \B\boxtimes_{\A}\M\to \C\boxtimes \M$, given by 
$$
\B\boxtimes_{\A}\M\xrightarrow{F\boxtimes_{\A} \Id_{\M}} \C\boxtimes \End(\M)\boxtimes_{\A}\M=\C\boxtimes \M\boxtimes \A_{\M}^{*\rm op}\xrightarrow{\Id_{\C}\boxtimes \rho} \C\boxtimes \M,
$$ 
is an equivalence (where $\rho:\M\boxtimes \A_{\M}^{*\rm op} \to \M$ is the right action of $\A_{\M}^{*\rm op}$ on $\M$). \qed 
\end{enumerate}
\end{theorem}  

Let $\mathscr{N}$ be an indecomposable exact module category over $\C$. 
Then $\mathscr{N}\boxtimes \mathscr{M}$ is an exact module category over 
$\C\boxtimes \Bnd(\M)$, and we have $(\C\boxtimes \Bnd(\M))_{\mathscr{N}\boxtimes \mathscr{M}}^*\cong\C_{\mathscr{N}}^*$. By \cite[Theorem 4.1]{EG}, the dual sequence to (\ref{ES}) with respect to $\N\boxtimes \M$ 
\begin{equation}\label{ES1}
\A_{\mathscr{M}}^*\boxtimes \Bnd(\mathscr{N})\xleftarrow{\iota^*} \B_{\mathscr{N}\boxtimes \mathscr{M}}^*\xleftarrow{F^*} \C_{\mathscr{N}}^*
\end{equation}
is exact with respect to $\N$.

\subsection{Lagrangian subcategories}
\label{Centralizers and Lagrangian subcategories}
 
Let $\D$ be a finite {\em braided} tensor category with braiding ${\rm c}$ (see, e.g., \cite[Chapter 8]{EGNO}), and let $\E\subset\D$ be a tensor subcategory of $\D$. Recall that two objects $X,Y\in\D$ {\em centralize} each other if 
\begin{equation}\label{squared braiding}
{\rm c}_{Y,X}\circ {\rm c}_{X,Y}= \Id_{X\otimes Y},
\end{equation}
and the (M\"{u}ger) {\em centralizer} of $\E$ is the tensor subcategory $\E'\subset\D$ consisting of all objects which centralize every object of $\E$ (see, e.g., \cite{DGNO}). The category $\D$ is called {\em non-degenerate} if $\D'=\Vect$. 

By \cite[Theorem 4.9]{S}, we have  
\begin{equation}\label{shimizu}
\FPdim(\E)\FPdim(\E')=\FPdim(\D)\FPdim(\D'\cap \E)
\end{equation}

If $\D$ is non-degenerate and $\E' = \E$, then $\E$ is called a {\em Lagrangian} subcategory of $\D$. Thus, a Lagrangian subcategory of $\D$ is a maximal symmetric tensor subcategory of $\D$. 

\subsection{Exact commutative algebras}\label{Exact commutative algebras}
Retain the notation from \ref{Centralizers and Lagrangian subcategories}. Let $A$ be an exact algebra in $\D$ (see \ref{exmodcat}), with multiplication and unit morphisms ${\rm m}_A$ and $u_A$, respectively. Recall that $A$ is {\em commutative} if ${\rm m}_A={\rm m}_A\circ {\rm c}_{A,A}$. 

Assume from now on that $A$ is an exact commutative algebra in $\D$. Consider the category $\C:=\Mod(A)_{\D}$ (see \ref{exmodcat}). Let $(X,{\rm m}_X)\in \C$, where $X\in\D$ and ${\rm m}_X:X\ot A\to X$ is the $A$-module structure morphism. 
Recall that the braiding on $\D$ defines on $X$ two structures of a {\em left} $A$-module as follows:
$$A\otimes X\xrightarrow{{\rm c}_{A,X}}X\otimes
A\xrightarrow{{\rm m}_X}X,\,\,\,\,\,\,A\otimes X\xrightarrow{{\rm c}_{X,A}^{-1}}X\otimes
A\xrightarrow{{\rm m}_X} X.$$ Both structures make $(X,{\rm m}_X)$ into an $A$-bimodule, denoted by $X_+$ and $X_-$, respectively. By identifying $\C$ with a subcategory of the finite tensor category ${\rm Bimod}_{\D}(A)$ (see \ref{exmodcat}) via the full embedding functor
$$
\C\xrightarrow{1:1} {\rm Bimod}_{\D}(A),\,\,\,(X,{\rm m}_X)\mapsto X_-,
$$
(see \cite[Section 8.8]{EGNO}), we see that $\C$ is a finite tensor category, with the tensor
product $\ot_A$. Namely, the tensor product of $X,Y\in \C$, denoted by $X\ot_A Y$, is given by $X\ot _A Y_-$ endowed with the structure map $\Id_X\ot {\rm m}_Y:X\ot _A Y_-\ot A\to X\ot _A Y_-$. 

Now let
\begin{equation}
\label{free module functor} F: \D \to \C,\,\,\, X\mapsto (X\otimes A,\Id_X\ot {\rm m}_A),
\end{equation}
be the {\em free $A$-module functor}, and let 
\begin{equation}
\label{adjoint to free module functor} I: \C \to \D,\,\,\, (X,{\rm m}_X)\mapsto X
\end{equation}
be the forgetful functor. 
Recall that $F,I$ are adjoint functors \cite[Lemma 7.8.12]{EGNO}, i.e., we have natural isomorphisms 
$$
\Hom_{\C}(F(X),Y)\cong \Hom_{\D}(X,I(Y)),\,\,\,X\in\D,\,Y\in\C.
$$ 

Recall that $F$ is a surjective {\em tensor} functor \cite{EGNO}, with tensor structure 
$J=\{J_{X,Y}\mid X,Y\in\D\}$, where 
$$J_{X,Y}:F(X\ot Y)\xrightarrow{\cong}F(X)\ot_A F(Y)$$ is the $\C$-isomorphism given by the composition
\begin{equation}\label{J}
X\ot Y\ot A\xrightarrow{\Id_{X}\ot u_A\ot \Id_{A}} X\ot A\ot_A Y\ot A.
\end{equation}

\begin{lemma}\label{dimmodca}
We have $\FPdim(\C)=\FPdim(\D)/\FPdim(A)$.
\end{lemma}

\begin{proof}
Follows from \cite[Lemma 6.2.4]{EGNO}.
\end{proof}

\subsection{Finite group schemes}\label{gsch} (See, e.g., \cite{W}.) Let $G$ be a finite group scheme over $k$, with  coordinate algebra $\O(G)$. Then $\O(G)$ is a finite dimensional commutative Hopf algebra, and 
its representation category ${\rm Coh}(G)$ is a finite tensor category over $k$.

We will denote the dimension of $\O(G)$ by $|G|$.

Let $\Rep(G)=\Corep(\O(G))$ denote the representation category of $G$ over $k$. Recall that $\Rep(G)$ is a finite {\em Tannakian} category, that is, a finite \emph{symmetric} tensor category such that the forgetful functor $\Rep(G)\to \Vect$ is symmetric.

Set $\E:=\Rep(G)$. Let ${\rm I}:\E\to \E\boxtimes\E$ be the adjoint functor of the tensor functor $\E\boxtimes\E\xrightarrow{\ot}\E$. Recall that ${\rm I}({\bf 1})$ is an exact commutative algebra in $\E\boxtimes\E$, and $\E\cong\Mod({\rm I}({\bf 1}))_{\E\boxtimes\E}$.

Let $\Delta:G\xrightarrow{1:1} G\times G$ be the diagonal morphism.

\begin{lemma}\label{gdiag}
We have ${\rm I}({\bf 1})=\O(G\times G/\Delta(G))$.
\end{lemma}

\begin{proof}
Since the tensor functor $\E\boxtimes\E\xrightarrow{\ot}\E$ coincides with the functor
$$\E\boxtimes\E=\Rep(G\times G)\xrightarrow{\Delta^*} \Rep(G)=\E,$$
the claim follows.
\end{proof}

\subsection{De-equivariantization}\label{deequiv}
Retain the notation of \ref{Centralizers and Lagrangian subcategories}-\ref{gsch}. Assume $\D$ contains  
$\E:=\Rep(G)$ as a Tannakian subcategory. 

Let $A:=\O(G)$. The group scheme $G$
acts on $A$ via {\em left} translations, making $A$ a {\em commutative}
algebra in $\E$. Thus, $A$ is a commutative
algebra in $\D$, with $\FPdim(A)=|G|$. 

Recall that $\Mod(A)_{\E}$ is equivalent to the standard $\E$-module category $\Vect$. Thus, $\Mod(A)_{\E}$ is exact, so $A$ is exact in $\E$. Thus by Theorem \ref{exactalg}, $A$ is exact in $\D$, so $\C:=\Mod(A)_{\D}$ and $\D^*_{\C}={\rm Bimod}_{\D}(A)^{{\rm op}}$ are finite tensor categories (see \ref{exmodcat}).
 
Let $\C^{0}=\Mod^{0}(A)_{\D}$ denote the subcategory of $\C$ consisting of all objects $(X,{\rm m}_X)$ such that ${\rm m}_X={\rm m}_X\circ {\rm c}_{A,X}\circ {\rm c}_{X,A}$ (i.e., dyslectic, or local, $A$-modules). Recall that $\C^{0}$ is a tensor subcategory of $\C$, which is moreover braided with the braiding inhereted from $\D$. 

\subsection{The categories ${\rm Coh}(G,\omega)$ and $\Z({\rm Coh}(G,\omega))$}\label{gscth} Retain the notation of \ref{gsch}, and let $\omega\in H^3(G,\mathbb{G}_m)$ be a normalized $3$-cocycle. That is, $\omega\in \mathscr{O}(G)^{\ot 3}$ is an
invertible element satisfying the equations
$$(\Id \ot \Id \ot \Delta)(\omega)(\Delta\ot \Id \ot \Id)(\omega)=(1\ot \omega)(\Id \ot \Delta\ot \Id )(\omega)(\omega\ot 1),$$
$$(\varepsilon\ot \Id \ot \Id )(\omega)=(\Id \ot \varepsilon\ot \Id)(\omega)=(\Id \ot \Id \ot \varepsilon)(\omega)=1.$$
Recall \cite[Section 5]{G2} that the category ${\rm Coh}(G,\omega)$ is just ${\rm Coh}(G)$ as abelian categories, equipped with the same tensor product, but with associativity constraint given by the action of $\omega$. Then ${\rm Coh}(G,\omega)$ is a finite {\em pointed} (i.e., every simple object is invertible) tensor category, and the group of its invertible objects is isomorphic to $G(k)$.

Let $\Z({\rm Coh}(G,\omega))$ denote the center of ${\rm Coh}(G,\omega)$ (see, e.g., \cite{EGNO}, \cite[Section 5]{G2}). By \cite[Theorem 4.2]{S}, $\Z({\rm Coh}(G,\omega))$ is a finite {\em non-degenerate} braided tensor category (see \ref{Centralizers and Lagrangian subcategories}). Recall moreover that $\Z({\rm Coh}(G,\omega))\cong \Rep(D^{\omega}(G))$ as braided tensor categories, where $D^{\omega}(G)$ is a quasi-Hopf algebra, called the twisted double of $G$.

Note that for every $\omega_1,\omega_2\in H^3(G,\mathbb{G}_m)$, we have
\begin{equation}\label{tenprofcenters}
\Z(\Coh(G,\omega_1))\boxtimes \Z(\Coh(G,\omega_2))\cong
\Z(\Coh(G\times G,\omega_1\times \omega_2))
\end{equation}
as braided tensor categories.

Let
$$\mathcal{F}:\Z({\rm Coh}(G,\omega))\twoheadrightarrow {\rm Coh}(G,\omega)$$
be the {\em forgetful functor}. The following lemma is well known. 

\begin{lemma}\label{centergomega0}
The tensor subcategory of $\Z({\rm Coh}(G,\omega))$ consisting of all objects mapped to $\Vect$ under $\mathcal{F}$ is tensor equivalent to $\Rep(G)$. Thus, $\Z({\rm Coh}(G,\omega))$ canonically contains $\Rep(G)$ as a Tannakian Lagrangian subcategory. \qed
\end{lemma}

Now let  
$$\mathcal{I}:{\rm Coh}(G,\omega)\to \Z({\rm Coh}(G,\omega))$$
be the adjoint functor of $\mathcal{F}$.

\begin{lemma}\label{centergomega}
The following hold:
\begin{enumerate}
\item
$\mathcal{I}({\bf 1})=\O(G)$ as objects of $\Rep(G)$.
\item
$\mathcal{I}({\bf 1})=\O(G)$ is an exact commutative algebra in $\Z({\rm Coh}(G,\omega))$.
\item
$\mathcal{I}$ induces a tensor equivalence 
$$\mathcal{I}:{\rm Coh}(G,\omega)\xrightarrow{\cong} \Mod(\O(G))_{\Z({\rm Coh}(G,\omega))}.$$
\item
$\mathcal{I}\circ \mathcal{F}$ coincides with the free $\O(G)$-module functor
$$F:\Z({\rm Coh}(G,\omega))\twoheadrightarrow \Mod(\O(G))_{\Z({\rm Coh}(G,\omega))},\,\,\,X\mapsto X\ot \O(G).$$
\end{enumerate}
\end{lemma}

\begin{proof}
(1) is clear, and since $\O(G)$ is an exact commutative algebra in $\Rep(G)$ (see \ref{deequiv}), (2) follows from Theorem \ref{exactalg}. The proof of (3), (4) is now similar to \cite[Lemma 8.12.2]{EGNO}.
\end{proof}

\subsection{The group $\mathcal{M}_{\rm ext}(\Rep(G))$}\label{mextrepg}
Fix a finite group scheme $G$ over $k$. Set $\E:=\Rep(G)$, and let ${\rm I}({\bf 1})\in \E\boxtimes\E$ be as in Lemma \ref{gdiag}. 

Consider the set of pairs $(\D,\iota)$, where $\D$ is a non-degenerate braided tensor category over $k$, and $\iota:\E\xrightarrow{1:1} \D$ is an injective braided tensor functor. We say that $(\D_1,\iota_1)$, $(\D_2,\iota_2)$ are equivalent, if there
exists a braided tensor equivalence $\psi:\D_1\xrightarrow{\cong}\D_2$ such that
$\psi\circ \iota_1= \iota_2$.
Let $\mathcal{M}_{\rm ext}(\E)$ denote the set of all equivalence classes of pairs $(\D,\iota)$.

Now for $(\D_1,\iota_1),(\D_2,\iota_2)$ in $\mathcal{M}_{\rm ext}(\E)$, let $B:=(\iota_1\boxtimes\iota_2)({\rm I}({\bf 1}))$. Then $B$ is an exact commutative algebra in $\D_1\boxtimes\D_2$. Consider the braided tensor category $\D:=\Mod^{0}(B)_{\D_1\boxtimes\D_2}$ (see \ref{deequiv}),
and let 
$\iota:\E\xrightarrow{1:1}\D$ be the injective functor 
$$\E=\Mod(I({\bf 1}))_{\E\boxtimes\E}\xrightarrow{\iota_1\boxtimes\iota_2}\Mod^{0}(B)_{\D_1\boxtimes\D_2}=\D.$$
Recall \cite{LKW} that the product rule 
$$(\D_1,\iota_1)\widetilde{\boxtimes}(\D_2,\iota_2)=(\D,\iota)$$
determines a commutative group structure on $\mathcal{M}_{\rm ext}(\E)$, 
with unit element $(\Z(\Coh(G)),\iota_0)$, where $\iota_0:\E\to \Z(\Coh(G))$ is the inclusion functor (see Lemma \ref{centergomega0}).

\section{The proof of Theorem \ref{mainnew}}\label{prrofmainnew}

Let $G$ be a finite $k$-group scheme, and let $\E:=\Rep(G)$ (see \ref{gsch}).

Let $\D$ be a finite {\em non-degenerate} braided tensor category over $k$, containing $\E$ as a Lagrangian Tannakian subcategory (see \ref{Centralizers and Lagrangian subcategories}). By (\ref{shimizu}), we have $\FPdim(\D)=|G|^2$ (see \ref{FTC}).

\subsection{The free module functor $F$} Let $A:=\O(G)$, and consider the $G$-de-equivariantization category $\C:=\Mod(A)_{\D}$  (see \ref{deequiv}). Since $A$ is exact in $\D$, $\C$ is a finite tensor category. By Lemma \ref{dimmodca}, we have 
\begin{equation}\label{thedimofc}
\FPdim(\C)=\FPdim(\D)/\FPdim(A)=|G|^2/|G|=|G|.
\end{equation}

Let $F:\D\to \C$ be the free $A$-module functor (\ref{free module functor}), and let $\E\xrightarrow{\iota}\D$ be the inclusion functor.

\begin{proposition}\label{restoE}
We have an exact sequence of tensor categories
\begin{equation}\label{exseqcenfun}
\E\xrightarrow{\iota}\D\xrightarrow{F}\C
\end{equation} 
with respect to the standard $\E$-module category $\Vect$ (see \ref{esftc}).
\end{proposition}

\begin{proof}
For every $\mathbf{V}=(V,\rho)\in \E$, we have $\mathbf{V}\ot A\cong V\ot A$ in $\E$. Thus, $F(\mathbf{V})\cong V\ot A$ lies in $\Vect\subset\C$ (as $A$ is the unit object of $\C$), so $F$ maps $\E$ to $\Vect$. Since by (\ref{thedimofc}), $\FPdim(\D)=\FPdim(\E)\FPdim(\C)$, the claim follows from Theorem \ref{eqdefexseq}.
\end{proof}

Recall that $\C^*_{\C}\cong \C^{\rm op}$ as tensor categories. Thus, dualizing (\ref{exseqcenfun}) with respect to the exact indecomposable $\D$-module category $\C$ \cite[Theorem 4.1]{EG} (see (\ref{ES1})), we obtain an exact sequence
$$\C^{\rm op}\xrightarrow{F^*} \D^*_{\C}\xrightarrow{\iota^*} \E^*_{\Vect}\boxtimes \Bnd(\C)$$
of finite tensor categories with respect to the indecomposable exact $\C^{\rm op}$-module category $\C$. Since $\Bnd(\C)\cong \Bnd(\C)^{\rm op}$ as monoidal categories, and $(\E^*_{{\rm Vec}})^{\rm op}=\E^*_{{\rm Vec}}\cong\Coh(G)$ as tensor categories (see, e.g., \cite[Example 7.9.11]{EGNO}), we get an exact sequence
\begin{equation}\label{exseqdual}
\C\xrightarrow{F^*} \D^{*{\rm op}}_{\C}\xrightarrow{\iota^*} \Coh(G)\boxtimes \Bnd(\C)
\end{equation}
of finite tensor categories with respect to the indecomposable exact $\C$-module category $\C$. 
In particular, $\C$ can be identified with a tensor subcategory of $\D^{*{\rm op}}_{\C}$ via $F^*$. Also, since $\D^{*{\rm op}}_{\C}\cong \Bimod_{\D}(A)$ as tensor categories (see \ref{exmodcat}), $\D^{*{\rm op}}_{\C}$ contains $\Coh(G)=\E^*_{\Vect}=\Bimod_{\E}(A)$ as a tensor subcategory.

\begin{lemma}\label{exfac123}
We have an equivalence of abelian categories
$$\Psi:\D^{*{\rm op}}_{\C}\xrightarrow{\cong} \Coh(G)\boxtimes \C$$
such that 
$$\Psi(X\ot Y)=X\boxtimes Y,\,\,\,X\in \Coh(G),\, Y\in \C.$$
In particular, $P_{\D^*_{\C}}({\bf 1})=P_{\Coh(G)}({\bf 1})\ot P_{\C}({\bf 1})$. 
\end{lemma}

\begin{proof}
Applying Theorem \ref{eqdefexseq}(3) to the exact sequence (\ref{exseqdual}) provides the natural equivalence of abelian categories  
$$T:\D^{*{\rm op}}_{\C}\boxtimes_{\C}\C\cong \Coh(G)\boxtimes \C,$$
and it is straightforward to verify 
that composing the natural equivalence $\D^{*{\rm op}}_{\C}\xrightarrow{\cong}\D^{*{\rm op}}_{\C}\boxtimes_{\C}\C$ with $T$ yields an abelian equivalence $\Psi$
with the stated property, as claimed.

In particular, it follows from the above that for every simple objects $X\in \Coh(G)$ and $Y\in \C$, the object $X\ot Y$ is simple in $\D^{*{\rm op}}_{\C}$, and that every simple object in $\D^{*{\rm op}}_{\C}$ is of this form. Furthermore, since $P_{\Coh(G)}(X)\boxtimes P_{\C}(Y)$ is the projective cover of $X\boxtimes Y$ in $\Coh(G)\boxtimes \C$, and  
$$\Psi(P_{\Coh(G)}(X)\ot P_{\C}(Y))=P_{\Coh(G)}(X)\boxtimes P_{\C}(Y),$$ it follows that $P_{\Coh(G)}(X)\ot P_{\C}(Y)$ is the projective cover of $X\ot Y$ in $\D^{*{\rm op}}_{\C}$. Thus, we get the last statement for $X=Y={\bf 1}$.
\end{proof}

\subsection{The central structure on $F$} For every $X\in \D$, $(Y,{\rm m}_Y)\in\C$, consider the morphism
\begin{equation}\label{centralstructure}
\Phi_{X,Y}:F(X) \ot_A Y \cong Y \ot_A F(X),
\end{equation}
given by the composition
\begin{eqnarray*}\label{centralstructure1}
\lefteqn{F(X) \ot_A Y=(X\ot A)\ot_A Y
\xrightarrow{\Id_X\ot ({\rm m}_Y\circ {\rm c}_{Y,A}^{-1})}X\otimes Y}\\ & & \xrightarrow{{\rm c}_{X,Y}}Y\otimes X\xrightarrow{\Id_Y\ot u_A\ot \Id_X} Y\ot A\ot X
\xrightarrow{\Id_{Y}\ot c_{A,X}}Y\ot (X\otimes A)\\& & \twoheadrightarrow Y \ot_A F(X).
\end{eqnarray*}

Using that $A$ is commutative, the following 
lemma can be verified in a straightforward manner.

\begin{lemma}\label{isomphixy}
The following hold:
\begin{enumerate}
\item
$\{\Phi_{X,Y}\mid\,X\in \D,\, Y\in\C\}$ is a natural family of $\C$-isomorphisms.
\item
For every $X,Z\in \D$ and $Y\in\C$, we have
$$\Phi_{X\ot Z,Y}=(\Id_Y\ot J^{-1}_{X,Z})\circ (\Phi_{X,Y}\ot\Id_{F(Z)})\circ (\Id_{F(X)}\ot\Phi_{Z,Y})\circ (J_{X,Z}\ot \Id_Y)$$
(see (\ref{J})).
\item
For every $X\in \D$ and $Y,Z\in\C$, we have
$$\Phi_{X,Y\ot_A Z}=(\Id_{Y}\ot\Phi_{X,Z})\circ(\Phi_{X,Y}\ot \Id_Z).\,\,\,\,\,\,\,\,\,\,\,\,\,\,\,\,\,\,
\,\,\,\,\,\,\,\,\,\,\,\,\,\,\,\,\,\,\,\,\,\,\,\,\,\,\,\,\,\,\,\,\,\,\,\,\,\,\,\,\,\,\,\,\,\,\,\,\qed$$
\end{enumerate}
\end{lemma}

In particular, Lemma \ref{isomphixy} implies the following corollary.

\begin{corollary}\label{fextendstobr}
The free $A$-module functor $F$ (\ref{free module functor}) extends to a {\em braided} tensor functor 
\begin{equation}\label{free module functor extension} 
{\rm F}: \D \to \Z(\C)
\end{equation}
such that $F$ coincides with the composition of
${\rm F}$ and the forgetful tensor functor $\Z(\C) \twoheadrightarrow \C$. \qed
\end{corollary}

Recall that $\C$ is an indecomposable exact module category over $\Z(\C)$
via the forgetful functor $\Z(\C) \twoheadrightarrow \C$, and $\Z(\C)^*_{{\C}}\cong \C\boxtimes \C^{{\rm op}}$ as tensor categories (see, e.g., \cite[Theorem 7.16.1]{EGNO}).

Recall that since $\D$ is non-degenerate, $\D\boxtimes \D^{{\rm op}}\cong \Z(\D)$ as braided tensor categories. Recall also, that $\Z(\D)\cong \Z(\D^*_{\C})$ as braided tensor categories \cite[Corollary 7.16.2]{EGNO}.

\begin{proposition} \cite[Proposition 4.2]{DGNO}
\label{tilde F}
The functor ${\rm F}$ (\ref{free module functor extension}) is an equivalence of braided tensor categories.
\end{proposition}

\begin{proof} Consider $\C$ as a module category over $\D$ and $\Z(\C)$
via $F$ and the forgetful functor $\Z(\C) \twoheadrightarrow \C$, respectively. In both cases, $\C$ is indecomposable and exact. Let 
\begin{equation}\label{functorT}
{\rm F}^*:\C\boxtimes \C^{{\rm op}}\to \D^*_{\C}
\end{equation}
be the dual functor to
${\rm F}$. Recall that $\D^*_{\C}\cong {\rm Bimod}_{\D}(A)^{{\rm op}}$ as tensor categories (see \ref{exmodcat}). Then it is straightforward to verify that we have 
\begin{equation}\label{explicitfstar}
{\rm F}^*(X\boxtimes Y)= X_+\otimes_AY_-,\,\,\,X,Y\in \C
\end{equation}
(see \ref{Exact commutative algebras}). 

In particular, we see
that the functor 
$$\D \boxtimes \D^{{\rm op}}\stackrel{F\boxtimes
F}{\longrightarrow} \C\boxtimes
\C^{{\rm op}}\stackrel{{\rm F}^*}{\longrightarrow}\D^*_{\C}={\rm Bimod}_{\D}(A)^{{\rm op}}$$ coincides with
the functor 
$$\D\boxtimes \D^{{\rm op}}\cong \Z(\D)\cong
\Z(\D^*_{\C})\twoheadrightarrow \D^*_{\C}.$$ Since the forgetful functor
$\Z(\D^*_{\C})\twoheadrightarrow \D^*_{\C}$ is surjective, we see that the functor ${\rm F}^*$ is surjective. Thus, 
${\rm F}$ is injective \cite[Theorem 7.17.4]{EGNO}. 

Finally, since ${\rm F}$ is an injective tensor functor between categories
of equal Frobenius-Perron dimension (see (\ref{thedimofc})), it is necessarily an equivalence \cite[Proposition 6.3.4]{EGNO}, as claimed.
\end{proof}

Following Proposition \ref{tilde F}, we now aim to prove that there exists $\omega\in H^3(G,\mathbb{G}_m)$ such that $\C\cong {\rm Coh}(G,\omega)$ as tensor categories.

\subsection{$\mathcal{O}(\C)=G(k)$} 
In this section we show that $\C$ is pointed, and its group of invertible objects is isomorphic to $G(k)$.

Recall that ${\rm Bimod}_{\E}(A)=\E^*_{\Vect}\cong {\rm Coh}(G)$ as tensor categories, and $\mathcal{O}({\rm Coh}(G))\cong G(k)$. Let $A_g$, $g\in G(k)$, with $A_e=A={\bf 1}$, denote the invertible objects of ${\rm Coh}(G)={\rm Bimod}_{\E}(A)\subset{\rm Bimod}_{\D}(A)$. Namely, $A_g=A$ as  right $A$-modules, and the left $A$-action is determined by $g\in G(k)$. 

Let $G(\C)$ denote the group of invertible objects of $\C$.

\begin{proposition}\label{grpisomo}
The following hold:
\begin{enumerate}
\item
The category $\C$ is pointed (i.e., $\mathcal{O}(\C)=G(\C)$).
\item
For every $X\in G(\C)$, we have
$$
\FPdim_{\C}(P_{\C}(X))=\FPdim_{\Coh(G)}(P_{\Coh(G)}({\bf 1}))
.$$
(See \ref{FTC}.)
\item
The tensor equivalence ${\rm F}^*$ (\ref{functorT}) induces a group isomorphism 
$$G(k)\xrightarrow{\cong} G(\C),\,\,\,g\mapsto X_g.$$
\item
The universal faithful group grading of $\C$ is given by 
$$\C=\bigoplus_{g\in G(k)}\C_g,$$ 
where $\C_g\subset \C$ denotes the smallest Serre subcategory of $\C$ containing the invertible object $X_g$ (see \cite[Section 4.14]{EGNO}). 
\end{enumerate}
\end{proposition}

\begin{proof}
(1) Let $X$ be simple in $\C$. By \cite[Proposition 6.3.1]{EGNO}, $X_+$ is simple in $\D^*_{\C}$ (see \ref{Exact commutative algebras}).
Since $X^*$ is simple in $\C$, and ${\rm F}^*$ (\ref{functorT}) is an equivalence, it follows from (\ref{explicitfstar}) that 
${\rm F}^*(X\boxtimes X^*)=X_+\ot_A (X^*)_-$ is simple in $\D^*_{\C}$. Since
$$
\Hom_{\D^*_{\C}}(X_+\ot_A (X^*)_-,A)=\Hom_{\D^*_{\C}}(X_+,X_+)=k,
$$
it follows that $X_+\ot_A (X^*)_-=A$ in $\D^*_{\C}$, hence $X\ot_A X^*=A$ in $\C$. Thus, $X$ is invertible in $\C$, as claimed.

(2) Since the equivalence ${\rm F}^*$ (\ref{functorT}) maps $P_{\D^*_{\C}}({\bf 1})$ to $P_{\C}({\bf 1})\boxtimes P_{\C}({\bf 1})$, it follows that  
$$\FPdim_{\D^*_{\C}}(P_{\D^*_{\C}}({\bf 1}))=\FPdim_{\C}(P_{\C}({\bf 1}))^2.$$
 
On the other hand, by Lemma \ref{exfac123}, we have
$$\FPdim_{\D^*_{\C}}(P_{\D^*_{\C}}({\bf 1}))=\FPdim_{\Coh(G)}(P_{\Coh(G)}({\bf 1}))\FPdim_{\C}(P_{\C}({\bf 1})).$$
Therefore, we get $\FPdim_{\C}(P_{\C}({\bf 1}))=\FPdim_{\Coh(G)}(P_{\Coh(G)}({\bf 1}))$.

Finally, since for every invertible object $X\in G(\C)$, 
$$\FPdim_{\C}(P_{\C}(X))=\FPdim_{\C}((P_{\C}({\bf 1})),$$  
the claim follows.

(3) Since ${\rm F}^*$ (\ref{functorT}) is an equivalence, there exist unique simple objects $X_g,Y_g\in\C$ such that 
$$(X_g)_+\ot_A (Y_g)_- ={\rm F}^*(X_g\boxtimes Y_g)=A_g.$$ Since $X_g,Y_g\in G(\C)$ by (1), it follows that $(X_g)_+=A_g\ot_A (Y_g)_+^*$, so forgetting the left $A$-module structure yields $X_g=Y_g^*$. Thus, the functor ${\rm F}^*$ (\ref{functorT}) induces an injective group homomorphism 
\begin{equation}\label{groupinjection}
G(k)\xrightarrow{1:1} G(\C),\,\,\,g\mapsto X_g.
\end{equation}

Now it follows from (1), (2) and \ref{thedimofc} that
\begin{eqnarray*}
\lefteqn{|G|=\FPdim(\C)}\\
& = & \sum_{X\in G(\C)}\FPdim(P_{\C}(X)) \ge 
\sum_{X_g\in G(\C)}\FPdim(P_{\C}(X_g))\\
& = & \sum_{g\in G(k)}\FPdim(P_{\Coh(G)}({\bf 1}))=\FPdim(\Coh(G))=|G|.
\end{eqnarray*}
Therefore, we have an equality
$$\FPdim(\C)=\sum_{X_g\in G(\C)}\FPdim(P_{\C}(X_g)),$$
which implies that the map (\ref{groupinjection}) is also surjective, as claimed.

(4) By (1), $\C_{{\rm ad}}$ is the smallest Serre tensor subcategory of $\C$ containing ${\bf 1}$ (see \cite[Section 4.14]{EGNO}). Thus, (3) implies that $G(k)$ is the universal group of $\C$, as claimed.
\end{proof} 

\begin{remark}\label{fusion}
Note that if $G$ is an abstract group such that $p$ does not divide $|G|$, then Theorem \ref{mainnew}, and hence Corollaries \ref{main1} and \ref{maincor}, follow already from Proposition  \ref{grpisomo} (see \cite[Theorem 4.5]{DGNO} and \cite[Theorem 4.22]{LKW}).
\end{remark}

\subsection{$\C\cong {\rm Coh}(G,\omega)$} In this section we show that $\C\cong {\rm Coh}(G,\omega)$ as tensor categories.

For every $(X,{\rm m}_X)\in\C$, $\mathbf{V}=(V,\rho)\in \Rep(G)\subset\D$, consider the $\C$-isomorphism
$$\Phi_{\mathbf{V},X}:F(\mathbf{V})\ot_A X\xrightarrow{\cong} X\ot_A F(\mathbf{V}),$$
given by the central structure on $F$ (\ref{centralstructure}):
\begin{equation}\label{newcentstr}
\Phi_{\mathbf{V},X}=(\Id_X\ot {\rm c}_{A,\mathbf{V}})\circ (\Id_X\ot u_A\ot \Id_{V})\circ {\rm c}_{\mathbf{V},X}\circ (\Id_{V}\ot ({\rm m}_X\circ {\rm c}_{X,A}^{-1})).
\end{equation} 
Since by Proposition \ref{restoE}, $F(\mathbf{V})=V$,
we obtain a family of $\C$-isomorphisms 
$$\Phi_{V,X}:V\ot X\xrightarrow{\cong} V\ot X,\,\,\,V\in\Vect,$$ given by
\begin{equation}\label{phixsimplicity}
V\ot X=\mathbf{V}\ot A\ot_A X\xrightarrow{\Phi_{\mathbf{V},X}}X\ot_A \mathbf{V}\ot A =X\ot V=V\ot X.
\end{equation}

Thus by (\ref{phixsimplicity}), we can view $\Phi_{V,X}$ as an element of 
$$\Hom_{\C}(X,(V^*\ot_k V)\ot X)\cong \Hom_{\C}(X,\End_k(V)\ot X),$$   
so by Lemma \ref{isomphixy}, and the fact that ${\rm End}(F_{\mid \E})=kG$ as Hopf algebras, we obtain a natural family $\Phi:=\{\Phi_{X}\mid X\in \C\}$ of $\C$-morphisms 
\begin{equation}\label{phix}
\Phi_{X}:(X,{\rm m}_X)\to (kG\ot X,\Id_{kG}\ot {\rm m}_X),
\end{equation}
or equivalently, $k$-algebra maps
\begin{equation}\label{phixnew}
\Phi_{X}:\O(G)\to \End_{\C}(X).
\end{equation}

\begin{lemma}\label{phicomod}
For every $X\in\C$, the morphism $\Phi_{X}$ (\ref{phix}) equips $X$ with a structure of a $kG$-comodule in $\C$ (equivalently, an $\O(G)$-module in $\C$).
\end{lemma}

\begin{proof}
By Proposition \ref{restoE} and Lemma \ref{isomphixy}(2), for every $V,U\in \Vect$ and $X\in\C$, we have 
$$\Phi_{V\ot U,X}=(\Phi_{V,X}\ot\Id_{U})\circ (\Id_{V}\ot\Phi_{U,X}).$$
Now it is straightforward to verify that this translates to the claim.  
\end{proof}

Now by Proposition \ref{grpisomo}(1) and \cite[Proposition 2.6]{EO}, $\C$ admits a quasi-tensor functor to $\Vect$. Let $Q:\C\to \Vect$ be such a functor. Then by Lemma \ref{phicomod}, for every $X\in\C$, we get a  $k$-algebra map   
$$\widetilde{\Phi}_X:\O(G)\xrightarrow{\Phi_X}{\rm End}_{\C}(X)\xrightarrow{Q} {\rm End}_{k}(Q(X)),$$ so we see that we have defined a functor 
\begin{equation}\label{phidef}
\widetilde{\Phi}:\C\to {\rm Coh}(G),\,\,\,X\mapsto (Q(X),\widetilde{\Phi}_X).
\end{equation}

\begin{lemma}\label{phicomodpart2}
For every $g\in G(k)$, $\widetilde{\Phi}_{X_g}=g$ (where $g$ is viewed as an $\O(G)$-module). Thus, $\widetilde{\Phi}$ induces an injective group homomorphism $G(\C)\xrightarrow{1:1} G(k)$.
\end{lemma}

\begin{proof}
The first claim follows from Proposition \ref{grpisomo} and Lemma \ref{phicomod}. Also, it is straightforward to verify that for every $g\in G(k)$, the map $\Phi_{V,X_g}\in \Aut_{\C}(V\ot X_g)$ is given by $\rho(g) \ot \Id_{X_g}$ for every $V\in\Vect$, which proves the second claim.  
\end{proof}

\begin{corollary}\label{qtenequ}
The functor $\widetilde{\Phi}$ (\ref{phidef}) is a quasi-tensor equivalence.
\end{corollary}

\begin{proof}
By Lemma \ref{isomphixy}(3), we have 
$$\Phi_{V,X\ot_A Y}=(\Id_{X}\ot\Phi_{V,Y})\circ(\Phi_{V,X}\ot \Id_Y)$$
for every $V\in \Vect$ and $X,Y\in\C$.  
Now it is straightforward to verify that this implies that 
$\widetilde{\Phi}(A)=(k,\widetilde{\Phi}_A)$ is the unit object (since ${\rm End}_{\C}(A)=k$, so $\widetilde{\Phi}_A:\O(G)\to k$ is the trivial homomorphism), and we have an isomorphism 
$$ 
\Phi(X\ot Y)=(X\ot Y,\Phi_{X\ot Y})\cong(X,\Phi_X)\ot (Y,\Phi_Y)=\Phi(X)\ot \Phi(Y)
$$
for every $X,Y\in\C$. Thus, $\Phi$ is a quasi-tensor functor, and hence so is $\widetilde{\Phi}$.

Moreover, by Proposition \ref{grpisomo}(3) and Lemma \ref{phicomod}, $\widetilde{\Phi}$ induces a  group isomorphism $G(\C)\xrightarrow{\cong} G(k)$, which implies it is an equivalence.
\end{proof}
 
Finally, by Corollary \ref{qtenequ}, the associativity structure on $\C$ determines an associativity structure on ${\rm Coh}(G)$, i.e., a class $\omega\in H^3(G,\mathbb{G}_m)$, such that $\widetilde{\Phi}:\C\to {\rm Coh}(G,\omega)$ is a {\em tensor} equivalence, as desired.

This completes the proof of Theorem \ref{mainnew}.

\section{The proof of Corollary \ref{main1}}\label{prrofmain1}

Consider a braided tensor equivalence ${\rm F}: \D\xrightarrow{\cong} \Z({\rm Coh}(G,\omega))$, and let $\mathcal{F}:\Z({\rm Coh}(G,\omega))\twoheadrightarrow {\rm Coh}(G,\omega)$ be the forgetful functor. Let $f(G, \omega, {\rm F})$ denote the subcategory of $\D$ consisting of all objects sent to $\Vect$ under $\mathcal{F}\circ{\rm F}$. In other words, $f(G,\omega,{\rm F})$ is the preimage of $\Rep(G)$ under ${\rm F}$ (see Lemma \ref{centergomega0}). Then $f(G,\omega,{\rm F})$ is a Tannakian Lagrangian subcategory of $\D$, which is clearly independent of the equivalence class of $(G,\omega,{\rm F})$. 
Thus, every equivalence class $[(G,\omega,{\rm F})]$ gives rise to a Tannakian Lagrangian subcategory ${\rm f}([G,\omega,{\rm F}]):=f(G,\omega,{\rm F})$ of $\D$.

Conversely, assume $\Rep(G)\subset \D$ is a Tannakian Lagrangian subcategory of $\D$ 
for some finite group scheme $G$ over $k$. Then by
Theorem \ref{mainnew}, we have a braided tensor equivalence ${\rm F}:\D \cong \Z({\rm Coh}(G,\omega))$ for some  
$\omega\in H^3(G,\mathbb{G}_m)$. Thus, every Tannakian Lagrangian subcategory $\Rep(G)$ of $\D$ gives rise to an equivalence class $[(G,\omega,{\rm F})]$. Set ${\rm h}(\Rep(G)):=[(G,\omega,{\rm F})]$.

We claim that the assignments ${\rm f},{\rm h}$ constructed above are inverse to each other.

Given a Tannakian Lagrangian subcategory $\E= \Rep(G)$ of $\D$, let $A:=\O(G)$ and $\C:=\Mod(A)_{\D}$ (see \ref{Exact commutative algebras}). Since by Corollary \ref{fextendstobr}, the functor $\D\xrightarrow{{\rm F}} \Z(\C)\xrightarrow{{\rm Forget}}\C$ coincides with the free $A$-module functor, it follows that the category ${\rm f}({\rm h}(\E))=f(G,\omega,{\rm F})$
consists of all objects $X$ in $\D$ such that $X \otimes A$ is a multiple of $A$.
Since $A$ is the regular object of $\E$, it follows that $\E\subset {\rm f}({\rm h}(\E))$, so  
$\E={\rm f}({\rm h}(\E))$ (as both categories have FP dimension $|G|$). Thus, ${\rm f}\circ {\rm h} ={\rm Id}$, as desired. 

Given an equivalence class $[(G_1,\omega_1,{\rm F}_1)]$, we have to show now that $({\rm h}\circ {\rm f})([G_1,\omega_1,{\rm F}_1])=[(G_1,\omega_1,{\rm F}_1)]$. 
Since $f(G_1,\omega_1,{\rm F}_1)\cong{\rm Rep}(G_1)$ as symmetric tensor categories (via ${\rm F}_1$), it follows from \cite{DM} that there exist a finite group scheme $G_2$ over $k$, and a group scheme isomorphism $\phi:G_1\xrightarrow{\cong}G_2$, such that $f(G_1,\omega_1,{\rm F}_1)={\rm Rep}(G_2)$ and the restriction of ${\rm F}_1$ to $f(G_1,\omega_1,{\rm F}_1)$ coincides with $\phi^*:\Rep(G_2)\xrightarrow{\cong}\Rep(G_1)$. 
Thus, we have 
$$({\rm h}\circ {\rm f})([G_1,\omega_1,{\rm F}_1])={\rm h}({\rm Rep}(G_2))=[(G_2,\omega_2,{\rm F}_2)],$$ 
where ${\rm F}_2:\D \cong \Z({\rm Coh}(G_2,\omega_2))$ is a braided tensor equivalence for some $\omega_2\in H^3(G_2,\mathbb{G}_m)$. Since ${\rm F}_1(\O(G_2))=\phi^*(\O(G_2))=\O(G_1)$, it follows that ${\rm F}_1$ induces a tensor equivalence 
$${\rm F}_1:{\rm Mod}(\O(G_2))_{\D}\xrightarrow{\cong}{\rm Mod}(\O(G_1))_{\Z({\rm Coh}(G_1,\omega_1))}.$$
Also, ${\rm F}_2$ induces a tensor equivalence 
$${\rm F}_2:{\rm Mod}(\O(G_2))_{\D}\xrightarrow{\cong}{\rm Mod}(\O(G_2))_{\Z({\rm Coh}(G_2,\omega_2))}.$$

Now, let 
$$\mathcal{I}_i:{\rm Coh}(G_i,\omega_i)\xrightarrow{\cong}{\rm Mod}(\O(G_i))_{\Z({\rm Coh}(G_i,\omega_i))},\,\,\,i=1,2$$ 
be the tensor equivalences given by Lemma \ref{centergomega}(3), and consider  
the tensor equivalence 
$$\varphi:=\mathcal{I}_2^{-1}\circ {\rm F}_2\circ {\rm F}_1^{-1}\circ \mathcal{I}_1:{\rm Coh}(G_1,\omega_1)\xrightarrow{\cong}{\rm Coh}(G_2,\omega_2).$$ 
Then we have $\mathcal{F}_2\circ {\rm F}_2=\varphi\circ \mathcal{F}_1\circ {\rm F}_1$, where 
$$\mathcal{F}_i:\Z({\rm Coh}(G_i,\omega_i))\twoheadrightarrow {\rm Coh}(G_i,\omega_i),\,\,\,i=1,2$$ are the forgetful functors. Thus, $[(G_1,\omega_1,{\rm F}_1)]=[(G_2,\omega_2,{\rm F}_2)]$, so we have ${\rm h}\circ {\rm f}=\Id$, as desired.

The proof of Corollary \ref{main1} is complete. \qed

\section{The proof of Corollary \ref{maincor}}\label{proofmaincor} 

Fix a finite group scheme $G$ over $k$. Set $\E:=\Rep(G)$, and retain the notation of \ref{mextrepg}.

By Theorem \ref{mainnew}, we have a surjective map of sets
$$\alpha:\mathcal{M}_{\rm ext}(\E)\cong H^3(G,\mathbb{G}_m),\,\,\,(\D,\iota)\mapsto \omega,$$
where $\omega\in H^3(G,\mathbb{G}_m)$ is such that $\D\cong \Z(\Coh(G,\omega))$ as braided tensor categories. Moreover, it is clear that $\alpha(\D,\iota)=1$ if and only if $\D\cong \Z(\Coh(G))$ as braided tensor categories. Thus, it remains to show that $\alpha$ is a group homomorphism.

Set $A:=\O(G\times G)$ and $B:=\O(G\times G/\Delta(G))$ (see Lemma \ref{gdiag}). Then $B\subset A$ are exact commutative algebras in the Tannakian category $\Rep(G\times G)=\Rep(G)\boxtimes \Rep(G)$, and we have to prove that
$${\rm Mod}^{0}(B)_{\Z(\Coh(G,\omega_1))\boxtimes\Z(\Coh(G,\omega_2))}\cong \Z(\Coh(G,\omega_1\omega_2))$$ as braided tensor categories. By (\ref{tenprofcenters}), it suffices to prove that
$${\rm Mod}^{0}(B)_{\Z(\Coh(G\times G,\omega_1\times \omega_2))}\cong \Z(\Coh(G,\omega_1\omega_2))$$ as braided tensor categories.

To this end, note that similarly to \cite[Example 4.11(ii)]{DMNO}, one shows that  exact subalgebras $B$ of $A$ are in one to one correspondence with tensor subcategories of $\Z(\Coh(G\times G,\omega_1\times \omega_2))$, such that $B$ corresponds to the image of the functor 
\begin{equation}\label{image}
{\rm Mod}^{0}(B)_{\Z(\Coh(G\times G,\omega_1\times \omega_2))}\xrightarrow{-\ot_{B}A} \Coh(G\times G,\omega_1\times \omega_2).
\end{equation} 
Since the image of the functor (\ref{image}) is 
$$\Coh(\Delta(G),(\omega_1\times \omega_2)_{\mid \Delta(G)})\cong\Coh(G,\omega_1\omega_2),$$ we are done. \qed

\section{Examples}\label{examples}

Recall that a finite group scheme $G$ over $k$ is \emph{commutative} if and only if $\O(G)$ is a finite dimensional commutative and {\em cocommutative} Hopf algebra. Thus, if $G$ is a finite commutative group scheme over $k$ then its \emph{group algebra} $kG:=\O(G)^*$ is also a finite dimensional commutative and cocommutative Hopf algebra, so it represents a finite \emph{commutative} group scheme $G^D$ over $k$, which is called the \emph{Cartier dual} of $G$.

\begin{example}\label{exampl1}
Let $G$ be a {\em finite abelian $p$-group}. Then by \cite[Corollary 5.8]{EG2}, $H^3(G^D,\mathbb{G}_m)=1$, thus by Corollary \ref{maincor}, $\mathcal{M}_{\rm ext}(\Rep(G^D))=1$ is the trivial group. For example, if $G=\mathbb{Z}/p\mathbb{Z}$ then $G^D=\mu_p$ is the Frobenius kernel of the multiplicative group $\mathbb{G}_m$ (see, e.g., \cite[Section 2.2]{G2}), so $\mathcal{M}_{\rm ext}(\Rep(\mu_p))=1$.
\end{example}

\begin{example}\label{exampl2}
Let $\alpha_{p,r}$ denote the {\em $r$-th Frobenius kernel of the additive group $\mathbb{G}_a$} (see, e.g., \cite[Section 2.6]{EG2}). 
Let $G:=\prod_{i=1}^n \alpha_{p,r_i}$. Then by \cite[Corollary 5.10]{EG2}, we have $H^3(G^D,\mathbb{G}_m)=H^3(G^D,\mathbb{G}_a)$. Thus by Corollary \ref{maincor}, we have a group isomorphism
$$\mathcal{M}_{\rm ext}(\Rep(G^D))=H^3(G^D,\mathbb{G}_a).$$

For example, if $G=\alpha_p^n$ then $(\alpha_p^n)^D=\alpha_p^n$, so we have    
$$\mathcal{M}_{\rm ext}(\Rep(\alpha_p^n))=H^3(\alpha_p^n,\mathbb{G}_a).$$ 
Thus by \cite[Proposition 2.2]{EG2}, we have
$$\mathcal{M}_{\rm ext}(\Rep(\alpha_p^n))\cong \wedge^3 \mathfrak{g} \oplus (\mathfrak{g}\ot \mathfrak{g}^{(1)}),\,\,\,p>2,$$
and 
$$\mathcal{M}_{\rm ext}(\Rep(\alpha_2^n))\cong S^3 \mathfrak{g},$$
where $\mathfrak{g}:={\rm Lie}(\alpha_p^n)$ and $\mathfrak{g}^{(1)}$ is the Frobenius twist of $\mathfrak{g}$ \footnote{Namely, $\mathfrak{g}^{(1)}=\g$ as abelian groups, and $a\in k$ acts on $\mathfrak{g}^{(1)}$ as $a^{p^{-1}}$ does on $\g$.}.
\end{example}

\begin{example}\label{exampl3}
Let $\g$ be a finite dimensional {\em restricted $p$-Lie algebra} over $k$, and let $u(\g)$ be its {\em restricted universal enveloping algebra} (see, e.g., \cite[Section 2.2]{G2}). Since $u(\g)^*$ is a finite dimensional commutative Hopf algebra over $k$, $u(\g)^*=\O(\Gamma)$ is the coordinate algebra of a finite group scheme $\Gamma$ over $k$. Recall that $\g$ is the Lie algebra of $\Gamma$, and $\Rep(\Gamma)\cong \Rep(u(\g))$ as symmetric tensor categories.

Now by \cite[Theorem 5.4]{EG2}, we have $H^3(\Gamma,\mathbb{G}_m)=H^3(\Gamma,\mathbb{G}_a)$. Thus by Corollary \ref{maincor}, we have a group isomorphism
\begin{equation}\label{sscoh}
\mathcal{M}_{\rm ext}(\Rep(\Gamma))=H^3(\Gamma,\mathbb{G}_a).
\end{equation}

For example, if $\g$ is {\em semisimple} (i.e., $\g$ is the Lie algebra of a simple, simply connected algebraic group defined and split over $\mathbb{F}_p$) then by \cite[Corollary 1.6]{FP}, $H^{\bullet}(\Gamma,\mathbb{G}_a)$ is zero in odd degrees. Thus, by (\ref{sscoh}), we have $\mathcal{M}_{\rm ext}(\Rep(\Gamma))=H^3(\Gamma,\mathbb{G}_a)=1$ is trivial in this case.
\end{example}

\begin{example}\label{exampl4}
Let $\g$ be a {\em semisimple} restricted $p$-Lie algebra over $k$ (see Example \ref{exampl3}). Let $\mathfrak{k}$ be the $1$-dimensional abelian restricted $p$-Lie algebra over $k$ with $u(\mathfrak{k})^*=\O(\alpha_p)=k[x]/x^p$, where $x$ is a primitive element \footnote{I.e., $\Delta(x)=x\ot 1+1\ot x$.}. Recall that 
$H^1(\alpha_p,\mathbb{G}_a)={\rm sp}_k\{x\}$.
 
Consider now the restricted $p$-Lie algebra $\tilde{\g}:=\g\oplus \mathfrak{k}$, and let $\widetilde{\Gamma}$ be the finite group scheme over $k$ such that $u(\tilde{\g})^*=\O(\widetilde{\Gamma})$. Then $\widetilde{\Gamma}=\Gamma\times \alpha_p$, and the cup product induces an isomorphism
$$H^2(\Gamma,\mathbb{G}_a)\ot H^1(\alpha_p,\mathbb{G}_a)\xrightarrow{\cong} H^3(\widetilde{\Gamma},\mathbb{G}_a).$$
Equivalently, since $x$ is a basis of $H^1(\alpha_p,\mathbb{G}_a)$, we have an isomorphism
$$H^2(\Gamma,\mathbb{G}_a)\xrightarrow{\cong} H^3(\widetilde{\Gamma},\mathbb{G}_a),\,\,\,\xi\mapsto \xi\ot x.$$

Now recall (see, e.g., \cite{FP}) that for $p\ne 2,3$, we have an isomorphism
$$\tau:\g^{*(1)}\xrightarrow{\cong} H^2(\Gamma,\mathbb{G}_a),\,\,\,f\mapsto \g_f,$$ where $\g^{*(1)}$ is the Frobenius twist of $\g^*$ \footnote{Namely, $\g^{*(1)}$ is the space of semi-linear maps $\mathfrak{g}\to k$.}, and $\g_f=\g\oplus k$ (as an abstract Lie algebra) with $[p]$-operator defined by $(v,a)^{[p]}:=(v^{[p]},f(v))$.
Thus, we obtain an isomorphism 
$$\Omega:\g^{*(1)}\xrightarrow{\cong}H^3(\widetilde{\Gamma},\mathbb{G}_a),\,\,\,f\mapsto \tau(f)\ot x.$$
In particular, by (\ref{sscoh}), we have 
$$\mathcal{M}_{\rm ext}(\Rep(\widetilde{\Gamma}))\cong \g^{*(1)}.$$  
\end{example}

\end{document}